\documentclass[12pt]{amsart}

\usepackage{amsmath}
\usepackage{amsfonts}
\usepackage[mathscr]{eucal}
\usepackage{amssymb}

\usepackage{color}

\theoremstyle{plain}
\newtheorem{thm}{Theorem}[section]
\newtheorem{cor}[thm]{Corollary} 
\newtheorem{lemma}[thm]{Lemma} 
\newtheorem{prop}[thm]{Proposition}

\theoremstyle{remark}

\newtheorem{remark}[thm]{Remark}

\theoremstyle{definition}

\newtheorem{defn}[thm]{Definition}

\newtheorem{ques}[thm]{Question}

\newtheorem{prob}[thm]{Problem}

\usepackage{enumerate}

\newcount\theTime
\newcount\theHour
\newcount\theMinute
\newcount\theMinuteTens
\newcount\theScratch
\theTime=\number\time
\theHour=\theTime
\divide\theHour by 60
\theScratch=\theHour
\multiply\theScratch by 60
\theMinute=\theTime
\advance\theMinute by -\theScratch
\theMinuteTens=\theMinute
\divide\theMinuteTens by 10
\theScratch=\theMinuteTens
\multiply\theScratch by 10
\advance\theMinute by -\theScratch

\def\today{{\number\day\space
 \ifcase\month\or
  January\or February\or March\or April\or May\or June\or
  July\or August\or September\or October\or November\or December\fi
 \space\number\year}}

\newcommand\Bfr{{\mathfrak B}}
\newcommand\CFnc{{C^*(\mathbb F(n,c))}}
\newcommand\CFnm{{C^*(\mathbb F(n,m))}}
\newcommand\Cpx{{\mathbb C}}
\newcommand\eps{\epsilon}
\newcommand\et{{\tilde e}}

\newcommand\Mcal{{\mathcal{M}}} 
\newcommand\Nats{{\mathbf N}}

\newcommand\qt{{\tilde q}}

\newcommand\smd[2]{\underset{#2}{#1}}
\newcommand\tr{{\mathrm{tr}}}
\newcommand\ut{{\tilde u}}

\newcommand{\cl}[1]{\mathcal{#1}}
\newcommand{\bb}[1]{\mathbb{#1}}

\begin{document}

\title[Synchronous correlation matrices]{Synchronous correlation matrices and Connes' embedding conjecture}

\author[Dykema]{Kenneth J.\ Dykema}
\address{K.\ Dykema, Department of Mathematics, Texas A\&M University,
College Station, TX 77843-3368, USA}
\email{kdykema@math.tamu.edu}
\thanks{The first author was supported in part by NSF grant DMS-1202660.}

\author[Paulsen]{Vern Paulsen}
\address{V.\ Paulsen, Department of Mathematics, University of Houston,
Houston, TX, USA}
\email{vern@math.uh.edu}

\date{24 March, 2015}

\begin{abstract} In \cite{PSSTW} the concept of synchronous quantum correlation matrices was introduced and these were shown
to correspond to traces on certain C*-algebras. In particular, synchronous correlation matrices arose in their study of
 various versions of quantum chromatic numbers of
graphs and other quantum versions of graph theoretic parameters. 
In this paper we develop these ideas further, focusing on the relations between synchronous correlation matrices 
and microstates. We prove that 
Connes' embedding conjecture is equivalent to the equality of two families of synchronous quantum correlation matrices. 
We prove that if Connes' embedding conjecture has a positive answer, then  the tracial rank and projective rank 
are equal for every graph. We then apply these results to more general non-local games. 
\end{abstract}

\maketitle

\section{Introduction}

The chromatic number of a graph is the minimum number of different colors required to color the vertices so that
no edge connects vertices of the same color.

In \cite{galliardwolf, AHKS06, cnmsw, SS12} the concept of the quantum chromatic number $\chi_{q}(G)$
of a graph $G$ was developed and inequalities for estimating this parameter,
as well as methods for its computation, were presented.
In \cite{pt_chrom} several new variants of the quantum chromatic number, especially, 
 $\chi_{qc}(G)$ and $\chi_{qa}(G)$, were introduced.  
The motivation behind these new chromatic numbers came from the 
conjectures of Tsirelson and Connes and the fact that the set of correlations of 
quantum experiments may possibly depend on which set of quantum mechanical axioms one 
chooses to employ.
 Given a graph $G$, the aforementioned chromatic numbers satisfy the inequalities
\[
  \chi_{qc}(G) \le \chi_{qa}(G)  \le \chi_{q}(G) \le \chi(G),
\]
where $\chi(G)$ denotes the classical chromatic number of the graph $G$.
If the strong form of Tsirelson's conjecture is true, then $\chi_{qc}(G) = \chi_{q}(G)$ for every graph $G$, while if 
Connes' Embedding 
Conjecture is true, 
then $\chi_{qc}(G) = \chi_{qa}(G)$ for every graph $G$.  
Thus, computing these invariants gives a means to test the corresponding conjectures.

The fractional chromatic number $\chi_{f}(G)$ of a graph $G$ is an important lower bound on $\chi(G).$ 
 D.~Roberson and L.~Man\v{c}inska~\cite{arxiv:1212.1724} introduced a non-commutative analogue of the fractional 
 chromatic number, which they called the 
{\it projective rank} of $G$,  denoted $\xi_{f}(G)$ 
and proved that $\xi_{f}(G)$ is a lower bound for $\chi_{q}(G)$.
However, it is still not known if $\xi_{f}(G)$ is a lower bound for 
the variants of the quantum chromatic number studied in \cite{pt_chrom}.

The paper \cite{PSSTW} introduced a new C*-algebra built from a graph and using traces on these algebras introduced a parameter $\xi_{tr}(G)$ which they called the 
\emph{tracial rank} and showed that it is a lower bound for $\chi_{qc}(G)$. 
They also gave a new interpretation of the projective rank and fractional chromatic number,
by proving that if one restricted the C*-algebras in the definition of the tracial rank to be 
either finite dimensional C*-algebras or abelian, then one obtained the projective rank and fractional chromatic number, 
respectively. 

In this paper, we prove that if Connes' embedding conjecture is true, then the tracial and projective ranks 
are equal for all graphs.

A key result of \cite{PSSTW}, that allowed the introduction of traces, was a correspondence between
certain quantum correlations, called ``synchronous correlations'' and
traces on a
particular
C*-algebra.
In this paper we further develop the theory of synchronous correlations. 
Using the equivalence of the microstates conjecture with Connes' embedding conjecture, we are able to prove that 
Connes' embedding conjecture is equivalent to equality of two families of sets of synchronous quantum correlations. 

In \cite{PSSTW}, it was noted that the graph theoretic parameters that they were studying all belonged to a family of two person games that they called ``synchronous games''.  We further develop the connection between traces and synchronous games.  In particular, we introduce the ``synchronous value'' of a two person game and show that it is equal to the supremum of
the values of
all tracial states on a fixed element  of a particular C*-algebra. 

\section{Quantum correlation matrices and Non-Local Games}
\label{sec:Qmats}

Imagine that two non-communicating players, Alice and Bob, receive inputs from some finite set $X$ of cardinality $n$ and produce 
outputs  belonging to some finite set $O$ of cardinality $m.$ 

The game $\cl G$ has ``rules'' given by a function
\[ \lambda: X \times X \times O \times O \to \{0,1 \} \]
where $\lambda(x,y,a,b) =0$ means that if Alice and Bob receive inputs $x,y,$ respectively, 
then producing respective outputs $a,b$ is ``disallowed''.

The game is {\bf synchronous} if whenever they receive the same input, 
they must produce the same output, i.e.,  $\lambda (x,x, a,b) = 0, \, \forall a \ne b.$

A {\bf  quantum strategy} for a game $\cl G$ means that  Alice and Bob have finite 
dimensional  Hilbert spaces $H_A$, $H_B$   and for each input  $x \in X$ 
Alice has a projective measurement $\{ E_{x,a} \}_{a \in O}$ on $H_A$, i.e., for each $x \in X$ Alice has a set of projections satisfying,
 $\sum_{a\in O}E_{x,a}=I$,
and, similarly, for each input $y \in X$ 
Bob has a projective measurement $\{ F_{y,b} \}_{b \in O}$ on $H_B$;
moreover,
they share a state $\psi \in H_A \otimes H_B,$ i.e., a unit vector.   In this case
\[ p(a,b|x,y) := \langle  E_{x,a} \otimes F_{y,b}  \psi, \psi \rangle \]
is the probability of getting outcomes $a,b$ given inputs $x,y.$ 

The set of $n \times m$ matrices of the form $\big( p(a,b|x,y) \big)$,
 arising from all choices of finite dimensional Hilbert spaces $H_A$ and $H_B$,  all projective measurements,
and all unit vectors,
is called the set of {\bf quantum correlation matrices}
and is, usually, denoted $Q(n,m).$
For
a slight improvement of 
notation 
from \cite{PSSTW} we set $C_q(n,m):= Q(n,m)$
and similarly for related sets of correlation matrices, as described below.

A quantum strategy is called
a {\bf winning quantum strategy} for the game if the probability of it ever producing a
disallowed output is 0. Thus, a winning quantum strategy
is
$\big( p(a,b|x,y) \big) \in C_q(n,m)$
such that
\[ \lambda(x,y,a,b) =0 \implies p(a,b|x,y) =0.\]

If the game is synchronous, then a winning quantum strategy must satisfy $p(a,b|x,x) =0$ for all $x$ and for all $a \ne b,$ and, consequently,
we call such a correlation tuple {\bf synchronous}.
We let $C^s_q(n,m)$ denote the set of all synchronous quantum correlation matrices.


By a {\bf commuting quantum strategy} we mean that there is a single
(possibly infinite dimensional)
Hilbert space $H$, 
and for each $x \in X$ Alice has a projective measurement $\{ E_{x,a} \}_{a \in O}$ on $H$,
and for each $y \in X$ Bob has a projective measurement $\{F_{y,b} \}_{b \in O} $ on $H$, 
satisfying $E_{x,a}F_{y,b} = F_{y,b}E_{a,x}, \, \forall x,y,a,b$  and they share a state $\psi \in H.$  
In this case the probabilities are given by
\[ p(a,b|x,y) = \langle E_{x,a}F_{y,b} \psi, \psi \rangle.\]
We let $C_{qc}(n,m)$ denote the set of {\bf commuting quantum correlation matrices},
namely, those $n \times m$ matrices $(p(a,b,|x,y))$ arising as described above,
and
we let
$C^s_{qc}(n,m)$ denote the subset of synchronous
commuting quantum correlation
matrices.
Clearly, $C_q(n,m)\subseteq C_{qc}(n,m)$.

The {\bf strong Tsirelson conjecture} is the conjecture that $C_q(n,m) = C_{qc}(n,m)$ for all $n,m.$

It is known that the set $C_{qc}(n,m)$ is closed but it is still not known if $C_q(n,m)$ is closed. So we set
$C_{qa}(n,m)$ equal to the closure of $C_q(n,m)$ and let $C^s_{qa}(n,m)$ denote the synchronous elements
of $C_{qa}(n,m).$

The {\bf weak  Tsirelson conjecture} is the conjecture that $C_{qa}(n,m) = C_{qc}(n,m)$ for all $n,m.$

Junge, Navascues, Palazuelos, Perez-Garcia, Scholtz and Werner \cite{jnppsw} proved that if Connes' embedding conjecture is true, then the 
weak Tsirelson conjecture is true. Recently, Ozawa \cite{oz} proved the converse, so we now know that the weak
Tsirelson conjecture and Connes' embedding conjecture are equivalent.


A {\bf classical strategy} or {\bf local strategy} is any commuting quantum strategy for which 
all the measurement operators commute. The set of these correlation matrices is generally denoted $LOC(n,m)$, but for consistency
of notation we denote these by $C_{loc}(n,m)$ and we let $C^s_{loc}(n,m)$ denote the synchronous matrices in this set.

In summary, we have four types of correlation matrices
\[ C_{loc}(n,m) \subseteq C_q(n,m) \subseteq C_{qa}(n,m) \subseteq C_{qc}(n,m),\]
together with their synchronous subsets
\[ C_{loc}^s(n,m) \subseteq C_q^s(n,m) \subseteq C_{qa}^s(n,m) \subseteq C_{qc}^s(n,m).\]

One important example of a synchronous game is the {\bf graph coloring game}. Given a graph $G=(V,E)$ on $n$ 
vertices where $V$ denotes the vertex set and $E \subseteq V \times V$ denotes the set of edges, we write $v \sim w$ whenever
$(v,w) \in E.$ In the graph coloring game the inputs are the vertices, i.e., $X=V$ and the outputs are a set of $c$ colors,
so without loss of generality, $O = \{ 1, \ldots, c \}.$  The rules are that whenever $v=w$ then Alice and Bob must both output the 
same color, so the game is synchronous, and whenever $v \sim w$ then they must output different colors.

The quantum coloring number, $\chi_x(G)$ for $x= loc, q, qa, qc$ is the least integer $c$ for which there exists a winning strategy
corresponding to a correlation matrix in $C^s_x(n,c).$ Interestingly, the classical chromatic number $\chi(G)$  is equal to
$ \chi_{loc}(G)$~\cite{pt_chrom}. 


In \cite[Theorem~5.4]{PSSTW} an important connection was made between synchronous correlation matrices and traces on C*-algebras.
Recall that a positive linear functional $\tau: \cl A \to \bb C$ on a unital C*-algebra $\cl A$ is called a {\bf tracial state} provided that $\tau(1) =1$ and $\tau(uv) = \tau (vu)$ for all $u,v \in \cl A.$
We summarize their result below.

\begin{thm}[\cite{PSSTW}]\label{synctr} $\big( p(a,b|x,y) \big) \in C^s_{qc}(n,m)$ if and only if there exists a unital C*-algebra $\cl A$ generated by projections
 $\{ e_{x,a} \}_{1 \le x \le n, 1 \le a \le m}$ satisfying $\sum_{a=1}^m e_{x,a} = 1, \, \forall x$ and a tracial state $\tau: \cl A \to \bb C$ such that
 \[ p(a,b|x,y) = \tau(e_{x,a}e_{y,b}), \forall x,y,a,b.\]
 Moreover, $\big( p(a,b|x,y) \big) \in C^s_q(n,m)$ if and only if the C*-algebra $\cl A$ can be taken to be finite dimensional,
 and $\big( p(a,b|x,y) \big) \in C^s_{loc}(n,m)$ if and only if the C*-algebra $\cl A$ can be taken to be abelian.
\end{thm}

\medskip

In order to bound quantum chromatic numbers,
the paper
\cite{PSSTW} introduced 
some parameters
of a graph, denoted $\xi_x,$ for $x \in \{ loc, q, qa, qc \}$ that had many nice properties,
including the fact that
they are
multiplicative for strong graph product.

Given a correlation matrix $\big( p(a,b|v,w) \big)$, the {\bf marginal probability} that Alice produces output $a$ given input $v$ is
\[ p_A(a,v) = \langle E_{v,a} \psi, \psi \rangle . \]
Note that
\[ p_A(a,v) = \sum_{b=1}^m p(a,b|v,w) ,\]
where the sum is independent of $w.$  Similarly, the marginal probability that Bob produces output $b$ given input $w$ is 
\[p_B(b|w) = \langle F_{w,b} \psi, \psi \rangle = \sum_{a=1}^m p(a,b|v,w).\]
 
When $\big( p(a,b|v,w) \big)$ is synchronous, it follows from Theorem~\ref{synctr} that $p_A(a|v) = p_B(a|v).$

\begin{defn}\label{def:xiX} 
Let $G$ be a graph on $n$ vertices.
    For $x \in \{loc, q, qa, qc \}$, let $\xi_{x}(G)$ be the infimum of the 
    positive real numbers
    $t$ such that there exists $\big(p(a,b|v,w) \big)_{v,a,w,b} \in C^s_{x}(n,2)$ 
    satisfying
\begin{gather*}
p_A(1|v)= p_B(1,v)= t^{-1}, \, \forall v, \\
v \sim w \implies p(1,1 | v, w)= 0.
\end{gather*}
\end{defn}

We have the following summary of the results in \cite{PSSTW}
(see \cite[Definition~5.9, Proposition~5.10, Theorem~6.8, Theorem~6.11]{PSSTW}).

\begin{thm}[\cite{PSSTW}] \label{xichar} Let $G$ be a graph on $n$ vertices. Then
$\xi_{qc}(G)$ is the reciprocal of the supremum of the set of all real numbers $\lambda$ for which there exists a unital C*-algebra
$\cl A$ generated by projections $\{ e_v \}_{v \in V}$ and a tracial state $\tau: \cl A \to \bb C$ satisfying:
\begin{align*}
 \tau(e_v) \ge \lambda, &\, \forall v,\\
 v \sim w \implies & e_ve_w =0.
\end{align*}
Moreover, if we require, in addition, that $\cl A$ be finite dimensional, then we obtain $\xi_{q}(G)$, which is equal to the Mancinska-Roberson
projective rank $\xi_f(G).$ If we require, in addition, that $\cl A$ be abelian, then we obtain $\xi_{loc}(G)$, which
is equal to the fractional chromatic number $\chi_f(G).$
\end{thm}

In the next section, we prove that if Connes' embedding is true, then $\xi_{qc}(G) = \xi_f(G)$ for every graph $G.$

\section{Microstates, correlation matrices and Connes' embedding conjecture}
\label{sec:CEC}

Connes' embedding conjecture is one of the most important outstanding problems in operator algebra theory.
Connes asked~\cite{C76} whether every II$_1$-factor having separable predual is embeddable in an ultrapower $R^\omega$
of the hyperfinite II$_1$-factor $R$.
This is, in essence, a question about approximation (in terms of moments) of elements of II$_1$ factors by matrices.
To use a term introduced by Voiculescu, if $(\Mcal,\tau)$ is a tracial von Neumann algebra
(namely, a von Neumann algebra $\Mcal$ with normal, faithful tracial state $\tau$)
and if $x_1,\ldots,x_n$ are self-adjoint elements of $\Mcal$,
we say that the tuple $(x_1,\ldots,x_n)$ {\em has matricial microstates} if for every $\eps>0$ and every
integer $N\ge1$, there is $k$ and there are $k\times k$ self-adjoint matrices $A_1,\ldots,A_n$ such that for all $p\le N$
and every $i_1,\ldots,i_p\in\{1,\ldots,n\}$, we have
\[
\big|\tr_k(A_{i_1}\cdots A_{i_p})-\tau(x_{i_1}\cdots x_{i_p})\big|<\eps.
\]
The set of matrices, $A_1, \ldots, A_n$ is called a $(N, \epsilon)$-matricial microstate for $x_1, \ldots, x_n.$

The following result is well known and follows quite easily from the definition of $R^\omega$; see, for example, \cite{V02}, p.\ 264, Remark~(d).
\begin{thm}\label{thm:embeddable}
For a countably generated, tracial von Neumann algebra $(\Mcal,\tau)$, the following are equivalent:
\begin{enumerate}[(i)]
\item $\Mcal$ is embeddable in $R^\omega$
\item all finite families of self-adjoint elements of $\Mcal$ have matricial microstates.
\end{enumerate}
Furthermore, assuming $\Mcal$ is finitely generated, the above conditions are also equivalent to the following condition:
\begin{enumerate}[(i)]
\setcounter{enumi}{2}
\item some finite generating set of $\Mcal$ (consisting of self-adjoint elements) has matricial microstates.
\end{enumerate}
\end{thm}


\begin{prop}\label{closureconnes}
If $C_{qc}^s(n,m)$ equals the closure of $C_q^s(n,m)$ for all $n,m,$ then Connes' embedding conjecture is true.
\end{prop}
\begin{proof}
By a result of Kirchberg~\cite{Ki93},
the truth of Connes' embedding conjecture is equivalent
to the ``unitary moments'' assertion,
which states that whenever $n\in\Nats$ and $u_1,\ldots,u_n$ are unitary elements of a II$_1$-factor $\Mcal$, (with tracial state denoted $\tau$)
and whenever $\eps>0$, there is $p\in\Nats$ and there are unitary matrices $U_1,\ldots,U_n\in M_p(\Cpx)$ such that
$|\tau(u_ku_\ell^*)-\tr_p(U_kU_\ell^*)|<\eps$ for all $k,\ell\in\{1,\ldots,n\}$.
(See~\cite{DJ11} for discussion of this slight modification of Kirchberg's formulation.)
We will observe that the ``unitary moments'' assertion follows if we assume $\overline{C_q^s(n,m)}=C_{qc}^s(n,m)$.

Let $u_1,\ldots,u_n$ be as above and let $\eps>0$.
Take an integer $m>6\pi/\eps$.
Let
\[
\ut_k=\sum_{j=1}^m\omega^je_{k,j}
\]
where $\omega=\exp(\frac{2\pi\sqrt{-1}}m)$ and
where $e_{k,j}$ is the spectral projection of the unitary $u_k$ for the arc
\[
\bigg\{\exp(2\pi t\sqrt{-1})\;\bigg|\;\frac{j-1}m\le t<\frac jm\bigg\}.
\]
Then
$\|\ut_k-u_k\|\le|1-\omega|<\eps/3$.
By hypothesis, there exist $p\in\Nats$ and projections $E_{k,j}$ in $M_p(\Cpx)$ such that
\[
\sum_{j=1}^mE_{k,j}=1\qquad(1\le k\le n)
\]
and
\[
\big|\tr_p(E_{k,i}E_{\ell,j})-\tau(e_{k,i}e_{\ell.j})\big|<\frac\eps{3m^2},\qquad(1\le k,\ell\le n,\;1\le i,j\le m).
\]
Let
\[
U_k=\sum_{j=1}^m\omega^jE_{k,j}.
\]
Then $U_k\in M_p(\Cpx)$ is unitary and $|\tr_p(U_kU_\ell^*)-\tau(\ut_k\ut_\ell^*)|<\eps/3$ for all $1\le k,\ell\le n$.
This implies $|\tr_p(U_kU_\ell^*)-\tau(u_ku_\ell^*)|<\eps$.
\end{proof}




As usual, we will denote by $\|\cdot\|_2$ the $2$-norm on $M_k(\bb C)$, given by $\|x\|_2=\tr_k(x^*x)^{1/2}$.

\begin{lemma}\label{lem:hdelta}
Let $\tau$ be a tracial state on a finite dimensional abelian C$^*$-algebra $A=\bb C^m$ and let $h\in A$
be a self-adjoint generator of $A$.
Let $\delta>0$.
Then there exists a positive integer $N$ and $\eps>0$ such that for all sufficiently large positive integers $k$,
if $a\in M_k(\bb C)$
is an $(N,\eps)$-microstate for $h$, then there is a unital $*$--representation $\pi:A\to M_k(\bb C)$
so that $\|\pi(h)-a\|_2<\delta$.
\end{lemma}
\begin{proof}
By Lemma~4.3 of~\cite{V94}, there is an integer $N>0$ and there is $\eps>0$ such that whenever
$a,b\in M_k(\bb C)$ are $(N,\eps)$-microstates for $h$, then there is a unitary $u\in M_k(\Cpx)$ such that
$\|a-ubu^*\|_2<\delta$.
For all $k$ sufficiently large, there is a unital $*$-representation $\pi:A\to M_k(\Cpx)$
so that $\pi(h)$ is an $(N,\eps)$-microstate for $h$.
Be Voiculescu's result, we can find unitary $u$ so that $\|a-u\pi(h)u^*\|<\delta$,
and replacing $\pi$ by $u\pi(\cdot)u^*$,
we are done.
\end{proof}

Let $\bb F(n,m)$ denote the free product of $n$ copies of the cyclic group of order $m,$ $\bb Z_m$ and let $\CFnm$ be the full group C$^*$-algebra.
Then $\CFnm$ is the universal unital free product C$^*$-algebra
\begin{equation}\label{eq:A}
\CFnm=*_1^n\Cpx^m
\end{equation}
of $n$ copies of the $m$--dimensional abelian C$^*$-algebra $\Cpx^m$.

\begin{defn}
Fix a set $H$ of self-adjoint elements of $\Cpx^m$, each of norm $\le1$, that generates $\Cpx^m$ as a unital algebra.
For every $j\in\{1,\ldots,n\}$,
let $H_j\subset\CFnm$ be the copy of $H$ in the $j$-th copy of $\Cpx^m$ in $\CFnm$, so that $\CFnm$ is generated by
$H_1\cup\cdots\cup H_n$.
For $\tau$ and $\sigma$ tracial states on $\CFnm$ and for $N\in\Nats$ and $\eps>0$, we will say that
{\em $\sigma$ approximates $\tau$ with tolerance $(N,\eps)$ for the generating set $H$}
if for every $p\in\{1,\ldots,N\}$ and $x_1,\dots,x_p\in H_1\cup\cdots\cup H_n$, we have
\[
\left|\tau(x_1\cdots x_p)-\sigma(x_1\cdots x_p)\right|<\eps.
\]
\end{defn}

\begin{remark}\label{rmk:HH'}
If $H'$ is another generating set of $\Cpx^m$ consisting of self-adjoint elements, then each element of $H'$ is a polynomial in elements of $H$.
Thus, 
for every $N'\in\Nats$ and $\eps'>0$, there are $N\in\Nats$ and $\eps>0$ such that if
$\sigma$ approximates $\tau$ with tolerance $(N,\eps)$ for the generating set $H$,
then
$\sigma$ approximates $\tau$ with tolerance $(N',\eps')$ for the generating set $H'$.
\end{remark}

\begin{prop}\label{prop:tauh}
Suppose Connes' embedding conjecture is true.
Let $H$ be a finite generating set for $\Cpx^m$,
consisting of self-adjoint elements.
Let $\tau$ be a tracial state on $\CFnm$.
Let $N\in\Nats$ and $\eps>0$.
Then there exists $k\in\Nats$ and a unital $*$--homomorphism $\pi:\CFnm\to M_k(\Cpx)$
so that the trace $\tr_k\circ\pi$ approximates $\tau$ with tolerance $(N,\eps)$
for the generating set $H$.
\end{prop}
\begin{proof}
In light of Remark~\ref{rmk:HH'}, we may without loss of generality assume $H$ is a singleton set, $H=\{h\}$,
and we write $H_j=\{h_j\}$.
Take $0<\delta<\eps/(2N)$.
Let $N_j\in\Nats$ and $\eps_j>0$ be obtained from Lemma~\ref{lem:hdelta}, so that for every
$(N_j,\eps_j)$-microstate $a_j\in M_k(\Cpx)$ for $h_j$,
there is a unital $*$-homomorphism $\pi_j:\Cpx^m\to M_k(\Cpx)$
with 
\begin{equation}\label{eq:piha}
\|\pi_j(h_j)-a_j\|_2<\delta.
\end{equation}
Let $N'=\max(N,N_1,\ldots,N_n)$ and $\eps'=\min(\eps/2,\eps_1,\ldots,\eps_n)$.
By the assumption that Connes' embedding conjecture is true, there exists
an $(N',\eps')$-microstate $(a_1,\ldots,a_n)$ for $(h_1,\ldots,h_n)$.
By the choice of $(N',\eps')$, there exist $*$-homomorphisms $\pi_j:\Cpx^m\to M_k(\Cpx)$ as above,
so that~\eqref{eq:piha} holds.
By the choice of $\delta$, it follows that $(\pi_1(h_1),\ldots,\pi_n(h_n))$ is an $(N,\eps)$-microstate for $(h_1,\ldots,h_n)$.
We have the universal free product $*$-homomorphism $\pi=*_1^n\pi_j:\CFnm\to M_k(\Cpx)$
and the previous statement implies that $\tr_k\circ\pi$ approximates $\tau$ with tolerance $(N,\eps)$
for the generating set $H$.
\end{proof}

\begin{thm}\label{connes=closure} Connes' embedding conjecture is true if and only if $C^s_{qc}(n,m)$ is the closure of $C^s_q(n,m)$ for all $n,m.$
\end{thm}
\begin{proof} Proposition~\ref{closureconnes} shows that equality of the closure implies that Connes' embedding conjecture is true. For the converse assume that Connes' embedding conjecture is true, and let $H= \{ e_1, \ldots, e_m \}$ be the coordinate projections for $\bb C^m$. Let $V$ be a set of cardinality $n,$ and for $v\in V$ let $H_v= \{ e_{v,1}, \ldots, e_{v,m} \}$ be a generating set for the $v$-th copy of $\bb C^m$ in $*_{v \in V} \bb C^m= \CFnm.$  

We know that the closure of $C^s_q(n,m)$ is a subset of $C^s_{qc}(n,m),$ so it is enough to show the reverse inclusion. 
Suppose that we are given $\big( p(i,j|v,w) \big) \in C^s_{qc}(n,m).$ By Theorem~\ref{synctr} there is a trace $\tau: \CFnm \to \bb C$ such that $p(i,j|v,w) = \tau(e_{v,i}e_{w,j}).$ Apply Proposition~\ref{prop:tauh} with $N=2$ to conclude that there is $k$ and a *-homomorphism $\pi: \CFnm \to M_k$ so that $\tr_k \circ \pi$ approximates $\tau$ 
with tolerance $(2, \epsilon)$ for $H.$
Hence,  \[ |\tr_k \circ \pi(e_{v,i}e_{w,j}) - p(i,j|v,w)| < \epsilon.\]
Let $E_{v,i} = \pi(e_{v,i}) \in M_k,$ so that these are projections and if we set
$p_{\epsilon}(i,j|v,w) = \tr_k(E_{v,i}E_{w,j}),$ then $\big( p_{\epsilon}(i,j|v,w) \big) \in C^s_q(n,m)$ and 
converges to $\big( p(i,j|v,w) \big)$ as $\epsilon \to 0.$ Hence, $C_{qc}^s(n,m)$ is contained in the closure of $C_q^s(n,m).$
\end{proof}

The above result characterizes the closure of $C^s_q(n,m)$, assuming that Connes' embedding conjecture is true. But can we say anything about the closure without assuming that the conjecture is true? In particular, we ask the following:

\begin{prob}[Synchronous Approximation Problem]\label{syncapprox} Is the closure of $C_q^s(n,c)$ equal to $C_{qa}^s(n,c)$ for all $n$ and $c$?
\end{prob}

If the answer to the above problem was affirmative, then it would give a new proof of Ozawa's result that Connes' embedding conjecture is true if and only if $C_{qc}(n,c)=C_{qa}(n,c)$ for all $n$ and $c.$ In fact, we would have that the following are equivalent:
\begin{enumerate}[(i)]
\item Connes' embedding conjecture is true,
\item $C^s_{qc}(n,c)=C^s_{qa}(n,c)$ for all $n,c,$
\item $C_{qc}(n,c)=C_{qa}(n,c)$ for all $n,c.$
\end{enumerate}

To see this, note that if the answer to Problem~\ref{syncapprox} is affirmative, then
the above result shows that (ii) implies (i).

The fact that (i) implies (iii) was proven in \cite{jnppsw, fritz2,
  fkpt_dg}. We sketch
the proof. By Kirchberg's result, if Connes' is true, then $\CFnc
\otimes_{\min} \CFnc = \CFnc \otimes_{\max} \CFnc$ for every $n,c$. The
matrices in $C_{qc}(n,c)$ are all given by $p(i,j|v,w) \phi(e_{v,i}
\otimes e_{w,j})$ for some state on $\CFnc \otimes_{\max} \CFnc.$ But
 $\phi$ is also a state on $\CFnc \otimes_{\min} \CFnc$ and all such
 states can be shown to be limits of states given by finite
 dimensional representations.  In fact, this was first explicitly
 shown in \cite{ws2008}.

It remains to show that (iii) implies (ii). But if the two sets are equal then their synchronous subsets are equal.


Now we consider a graph $G$ having vertex set $V$ consisting of $n$ vertices,
and without loops (so that every edge has two distinct vertices).
For $v,w\in V$, we will write $v\sim w$ when $v$ is connected to $w$ by an edge.
Let us fix $m\in\Nats$ and consider the C$^*$-algebra $\CFnm$ as in~\eqref{eq:A}, but written
\[
\CFnm=*_{v\in V}\Cpx^m
\]

Let $e_1,\ldots,e_m$ be the minimal  projections
in $\Cpx^m$ and for $v\in\Gamma_0$, let $e_{v,1},\ldots,e_{v,m}$ be the copies of these in the corresponding
generating copy of $\Cpx^m$ in $\CFnc$.
We will say that a tracial state $\tau$ on $\CFnm$ satisfies
\begin{itemize}
\item[$\bullet$]
the {\em orthogonality condition} if $\tau(e_{v,i}e_{w,i})=0$
whenever $v,w\in V$, $v\sim w$ and $i\in\{1,\ldots,m\}$
\item[$\bullet$]
the {\em weak orthogonality condition} if $\tau(e_{v,1}e_{w,1})=0$ whenever
$v,w\in V$ and $v\sim w$.
\end{itemize}
Note that the weak orthogonality condition depends on our choice of ordering of the projections;
thus, we fix such an ordering.
In practice, we will only be concerned with the weak orthogonality condition when $m=2$.

Our next main goal is the following result.
\begin{prop}\label{prop:WOC}
Suppose Connes' embedding conjecture is true.
Let $m=2$, consider the generating set $H=\{e_1\}$ for $\Cpx^2,$ 
let $N\in\Nats$ and let $\eps>0$.
Suppose $\tau$ is a tracial state on $\CFnm$ that satisfies the weak orthogonality condition.
Then there exists $k\in\Nats$ and a unital $*$-homomorphism $\pi:\CFnm\to M_k(\Cpx)$ such that
the trace $\tr_k\circ\pi$ satisfies the weak orthogonality condition and
approximates $\tau$ with tolerance $(N,\eps)$ for the generating set $H$.
\end{prop}
Once we have proven the above result we see that:

\begin{cor} Suppose that Connes' embedding conjecture is true and let $G$ be a graph.  Then $\xi_q(G) = \xi_{qc}(G),$ that is, the Mancinska-Roberson projective rank of $G$ is equal to the tracial rank of $G.$
\end{cor}
\begin{proof}
Applying Theorem~\ref{xichar}, we see that $\xi_{qc}(G)$ is the reciprocal of the largest $\lambda$ for which there exists a trace $\tau$ and projections $\{e_v \}_{v \in V}$ satisfying the weak orthogonality conditions, such that $\tau(e_v) \ge \lambda$ for all $v.$ But by the above result, whenever this happens, then for every $\epsilon >0$ there is a $k$ and projections $E_v \in M_k$ satisfying the weak orthogonality conditions with $tr_k(E_v) \ge \lambda - \epsilon.$

Thus, $\xi_q(G) \le \xi_{qc}(G).$ But since $C^s_q(n,2) \subseteq C^s_{qc}(n,2)$ the other inequality follows.
\end{proof}


For the next two lemmas we let $\Mcal$ be a finite von Neumann algebra equipped with a normal,
faithful tracial state $\tau$, and we let $\|x\|_2=\tau(x^*x)^{1/2}$ for $x\in\Mcal$ be the corresponding $2$-norm.
(Recall that $\Mcal$ is said to be a factor if its center is trivial;  for example matrix algebras $M_k(\Cpx)$ are factors.)
In fact, we will apply the lemmas only in the case of $\Mcal$ being a matrix algebra,
but it seems just as easy and possibly useful to write the result in greater generality.
\begin{lemma}\label{lem:qqt}
Let $\Mcal$ be a von Neumann algebra with normal, faithful tracial state $\tau$ and with projections $p,q\in\Mcal$.
Let $\delta=\tau(pq)$.
Then there is a unitary $u\in\Mcal$
and there is a projection $q'\in\Mcal$ such that
\begin{enumerate}[(i)]
\item $q'\perp p$
\item $q'\le u^*qu$
\item $\tau(q)-\tau(q')\le\delta$
\item $\|u-1\|_2\le2\sqrt\delta$
\item $\|q-q'\|_2\le5\sqrt\delta$.
\end{enumerate}
Suppose, furthermore, that $\Mcal$ is either diffuse (i.e., has no minimal projections) or is a finite factor (i.e., a matrix algebra $M_k(\Cpx)$ for some $k$).
Then there is a projection $\qt\in\Mcal$ such that 
\begin{enumerate}[(i)]
\setcounter{enumi}{5}
\item $q'\le\qt$
\item $\qt\perp p$
\item $\tau(\qt)=\min(\tau(q),1-\tau(p))$
\item $\|q-\qt\|_2\le6\sqrt\delta$.
\end{enumerate}
\end{lemma}
\begin{proof}
Note that we have $\delta\le1$.
To find $q'$ satisfying (i)-(v),
we may without loss of generality assume $\Mcal$ is generated by $\{1,p,q\}$.
As is well known, the universal, unital C$^*$-algebra $\Bfr$ generated by two projections $P$ and $Q$ is
the set of all continuous functions $f$ from $[0,1]$ into $M_2(\Cpx)$ whose values at the endpoints are diagonal,
where $P$ 
and $Q$ are represented by the functions
\[
P=\left(\begin{matrix} 1&0\\ 0&0\end{matrix}\right),\qquad
Q(t)=\left(\begin{matrix} t&\sqrt{t(1-t)}\\ \sqrt{t(1-t)}&1-t\end{matrix}\right).
\]
Furthermore, every tracial state $\sigma$ on $\Bfr$ is given by
\[
\sigma(f)=a_0f(0)_{11}+b_0f(0)_{22}+\int\tr_2(f(t))\,d\mu(t)+a_1f(1)_{11}+b_1f(1)_{22},
\]
for a Borel measure $\mu$ on the open interval $(0,1)$ and for nonnegative $a_0,b_0,a_1,b_1$, so that
$a_0+b_0+\mu((0,1))+a_1+b_1=1$.
Thus, the von Neumann algebra $\Mcal$ is the weak closure of the image of $\Bfr$ under the Gelfand--Naimark--Segal representation
of such a trace.
We get
\begin{equation}\label{eq:Mcal}
\Mcal=\smd{\Cpx}{a_0}\oplus\smd{\Cpx}{b_0}\oplus\big(L^\infty(\mu)\otimes M_2(\Cpx)\big)\oplus\smd{\Cpx}{a_1}\oplus\smd{\Cpx}{b_1},
\end{equation}
where $L^\infty(\mu)\otimes M_2(\Cpx)$ should be removed if $\mu$ is the zero measure, and is otherwise interpreted as being functions from $(0,1)$ into $M_2(\Cpx)$,
up to equivalence $\mu$-a.e.
The $a_i$ and $b_i$ are written in~\eqref{eq:Mcal} only to remind us about the trace.
To wit, we have
\[
\tau(r_0\oplus s_0\oplus f \oplus r_1\oplus s_1)=a_0r_0+b_0s_0+\int\tr_2(f(t))\,d\mu(t)+a_1r_1+b_1s_1.
\]
Of course, if any $a_i=0$ or $b_i=0$, then the corresponding summand in~\eqref{eq:Mcal} should be removed.
We also have
\[
\begin{matrix}
p&=&1\oplus0\,\oplus&\begin{pmatrix}1&0\\0&0\end{pmatrix}&\oplus\,1\oplus0 \\[4ex]
q&=&0\oplus1\,\oplus&\left(\begin{smallmatrix} t&\sqrt{t(1-t)}\\ \sqrt{t(1-t)}&1-t\end{smallmatrix}\right)&\oplus\,1\oplus0.
\end{matrix}
\]
We calculate
$\delta=\tau(pq)=\frac12\int t\,d\mu(t)+a_1$.
Letting
\[
q'=0\oplus1\oplus\begin{pmatrix}0&0\\0&1\end{pmatrix}\oplus0\oplus0,
\]
we have $q'\perp p$ and $\tau(q)-\tau(q')=a_1\le\delta$.
Moreover, we see that $q'$ is a subprojection of $u^*qu$ for the unitary
\[
u=1\oplus1\oplus\left(\begin{smallmatrix} \sqrt{1-t}&\sqrt{t}\\ -\sqrt{t}&\sqrt{1-t}\end{smallmatrix}\right)\oplus1\oplus1
\]
and we calculate
\[
\tau(|u-1|^2)=2\int\big(1-\sqrt{1-t}\big)\,d\mu(t)\le2\int t\,d\mu(t)\le4\delta,
\]
so~(iv) holds.
Now~(v) follows from (ii)--(iv).

We will now find $\qt$ satisfying (vi)--(viii).
If $1-\tau(p)\le\tau(q)$, then we simply let $\qt=1-p$.
If $\tau(q)<1-\tau(p)$, then
will let $\qt=q'+r$ for a projection $r\le(1-p)\wedge(1-q')$ such that $\tau(r)=\tau(q)-\tau(q')$.
Since $q'\le1-p$, $(1-p)\wedge(1-q')$ is a projection in $\Mcal$ of trace $1-\tau(p)-\tau(q')$;
moreover, the desired trace value, namely $\tau(q)-\tau(q')$, is less than $1-\tau(p)-\tau(q')$.
Now $\Mcal$ does contain a projection of trace $\tau(q)-\tau(q')$, namely, the projection $u^*qu-q'$.
Thus, assuming either $\Mcal$ is diffuse or a matrix algebra, we conclude that the desired projection $r$ exists.

Since $\tau(r)\le\delta$, we have $\|r\|_2\le\sqrt\delta$ and~(ix) follows from~(v).
\end{proof}

\begin{lemma}\label{lem:WOC}
Fix a graph $G$ as described above.
For every $\eps>0$ there is $\delta>0$ such that if
$(e_v)_{v\in V}$ are projections in $\Mcal$ satisfying
\begin{equation}\label{eq:tauee}
\tau(e_ve_w)<\delta,\quad(v,w\in V,\,v\sim w),
\end{equation}
then there exist projections
$(\et_v)_{v\in V}$ in $\Mcal$ satisfying
\begin{gather}
\et_v\perp\et_w,\quad(v,w\in V,\,v\sim w) \label{eq:eperpet} \\
\|e_v-\et_v\|_2<\eps\quad(v\in V). \label{eq:enearet}
\end{gather}
\end{lemma}
\begin{proof}
We proceed by induction on the number $n=|V|$ of vertices of the graph.
For $n=1$ there is nothing to prove, for we may take $\et_v=e_v$.
Suppose $n\ge2$ and the lemma has been proved for all smaller graphs.
Choose any vertex $v_0\in V$ and let $G'$ be the graph obtained from $G$
by removing the vertex $v_0$ (and all edges containing $v_0$).
We let $V'= V\backslash\{v_0\}$ denote the vertex set of $G'$.
Choose any $\eta$ satisfying
$0<\eta<\eps^2/(50(n-1))$.
By induction hypothesis, there is $\delta'>0$ such that whenever
$(e_v)_{v\in V'}$ are projections in $\Mcal$ satisfying
\[
\tau(e_ve_w)<\delta',\quad(v,w\in V',\,v\sim w),
\]
then there exist projections
$(\et_v)_{v\in V'}$ in $\Mcal$ satisfying
\begin{gather}
\et_v\perp\et_w,\quad(v,w\in V',\,v\sim w) \\
\|e_v-\et_v\|_2<\eta\quad(v\in V'). \label{eq:eet'}
\end{gather}
Let $\delta=\min(\delta',\eps^2/(50(n-1)))$ and
suppose $(e_v)_{v\in V}$ are projections in $\Mcal$ satisfying~\eqref{eq:tauee}.
Let $(\et_v)_{v\in V'}$ be projections obtained using the induction hypothesis as described above.
Then using also~\eqref{eq:eet'} we get
\[
\tau(e_{v_0}\et_w)<\delta+\eta,\quad(w\in V',\,v_0\sim w).
\]
Let
\[
f=\bigvee_{w\in V',\,v_0\sim w}\et_w.
\]
Then $f\le\sum_{v_0\sim w\in V'}\et_w$, so
$\tau(e_{v_0}f)=\tau(e_{v_0}fe_{v_0})<(n-1)(\delta+\eta)$.
By Lemma~\ref{lem:qqt}, there is a projection $\et_{v_0}\in\Mcal$ such that $\et_{v_0}\perp f$ and
\[
\|e_{v_0}-\et_{v_0}\|_2\le5\sqrt{(n-1)(\delta+\eta)}<\eps.
\]
This finishes the construction of the family $(\et_v)_{v\in V}$ of projections satisfying~\eqref{eq:eperpet} and~\eqref{eq:enearet}.
\end{proof}

\begin{proof}[Proof of Proposition~\ref{prop:WOC}]
Let $N'=\max(N,2)$.
Let $\delta>0$ be as obtained from Lemma~\ref{lem:WOC}, but for $\eps/2$ instead of $\eps$.
Let $\eps'=\min(\eps/2,\delta)$.
By Proposition~\ref{prop:tauh}, there is $k$ and a $*$-homomorphism $\rho:\CFnc\to M_k(\Cpx)$ such that $\tr_k\circ\rho$ approximates
$\tau$ with tolerance $(N',\eps')$ for the generating set $H$.
Consider the projection $E_v=\rho(e_{v,1})\in M_k(\Cpx)$.
Since $\tau$ was assumed to satisfy the weak orthogonality condition, we have $\tr_k(E_vE_w)<\eps'$
whenever $v,w\in V$ and $v\sim w$.
Using $\|X\|_2=\tr_k(X^*X)^{1/2}$ for $X\in M_k(\Cpx)$,
by Lemma~\ref{lem:WOC}, there exist projections $(E'_v)_{v\in\Gamma_0}$, such that
\begin{gather*}
E'_v\perp E'_w,\quad(v,w\in V,\,v\sim w) \\
\|E'_v-E_v\|_2<\frac\eps2\quad(v\in V).
\end{gather*}
Now defining $\pi:\CFnc\to M_k(\Cpx)$ to be the unital $*$-homomorphism determined by $e_v\mapsto E_v'$,
we have that $\tr_k\circ\pi$ approximates $\tau$ with tolerance $(N,\eps)$ for the generating set $H$.
\end{proof}

\begin{lemma}\label{lem:OC<1}
Fix a graph $G$ as described above and let $m\in\Nats$.
For every $\eps>0$ there is $\delta>0$ such that if
$(e_{v,i})_{v\in V,\,1\le i\le m}$ are projections in $\Mcal$ satisfying
\begin{align}
\tau(e_{v,i}e_{w,i})&<\delta,\quad(v,w\in V,\,v\sim w,\,1\le i\le m), \label{eq:taueiei} \\
\sum_{i=1}^me_{v,i}&\le 1,\quad(v\in V),
\end{align}
then there exist projections
$(\et_v)_{v\in V}$ in $\Mcal$ satisfying
\begin{gather}
\et_{v,i}\perp\et_{w,i},\quad(v,w\in V,\,v\sim w,\,1\le i\le m) \label{eq:eiperpeit} \\
\sum_{i=1}^m\et_{v,i}\le 1,\quad(v\in V )\\[1ex]
\|e_{v,i}-\et_{v,i}\|_2<\eps\quad(v\in V,\,1\le i\le m). \label{eq:eineareit}
\end{gather}
\end{lemma}
\begin{proof}
This follows immediately from Lemma~\ref{lem:WOC} applied to the graph $G^{(m)}$ obtained from $G$
as follows.
The vertex set $V(G^{(m)})$ is $V\times\{1,\ldots,m\}$.
There is an edge between vertices $(v,i)$ and $(w,j)$ in $G^{(m)}$ if and only if either (a) $v=w$  and $i\ne j$
or (b) there is an edge between $v$ and $w$ in $G$ and $i=j$.
\end{proof}

For those familiar with products of graphs, if $K_m$ denotes the complete graph on $m$ vertices, then $G^{(m)} = G \Box K_m,$ which is often called the {\it Cartesian product} of the graphs.
\begin{remark} If we could prove that the projections $\{ \tilde{e}_{v,i} \}$ can also be chosen to satisfy $\sum_{i=1}^m \tilde{e}_{v,i} = 1,$ for all $v \in V,$ then the above results would imply that Connes' embedding conjecture true implies that $\chi_q(G) = \chi_{qc}(G).$
\end{remark}

\begin{ques}
Fix a graph $G$ and a rational number $\gamma$.
Is it true that
for every $\eps>0$ there is $\delta>0$ such that for all integers $k$ that are large enough and divisible by the denominator of $\gamma$,
if $(e_v)_{v\in V}$ are projections in $M_k(\Cpx)$ satisfying
\begin{gather*}
\tr_k(e_ve_w)<\delta,\quad(v,w\in V,\,v\sim w), \\
\tr_k(e_v)=\gamma,\qquad(v\in V),
\end{gather*}
then there exist projections
$(\et_v)_{v\in V}$ in $M_k(\Cpx)$ satisfying
\begin{gather*}
\et_v\perp\et_w,\quad(v,w\in V,\,v\sim w)  \\
\|e_v-\et_v\|_2<\eps\quad(v\in V). \\
\tr_k(\et_v)=\gamma.
\end{gather*}
\end{ques}

\section{Values of Games}
Let $\cl G$ be a finite input-output game of the type described in the
introduction with inputs $X$, outputs $O,$ and with ``rules''
$\lambda: X \times X \times O \times O \to \{0,1\}.$  Suppose that in
addition we are given a probability distribution on the inputs. By
this we mean a set $\Gamma =( \gamma_{x,y}), $ such that $\gamma_{x,y}
\ge 0$ and $\sum_{x,y \in X} \gamma_{x,y} =1.$
Then for $t \in \{ loc, q, qa, qc \}$ we define the {\it value of the
  game given the distribution} to be
\begin{multline*}  \omega_t(\cl G, \Gamma) = \\ \sup \{\sum_{x,y \in X, i,j \in O}
\gamma_{x,y} \lambda(x,y,i,j) p(i,j|x,y) : \big( p(i,j|x,y) \big) \in C_t(n,c)\}.\end{multline*}

For $t \in \{ loc, qa, qc \}$ this supremum is actually attained, but,
since we do not know if $C_q(n,c)$ is closed the supremum is necessary
for this case. A major problem in the theory of non-local games,
related to the strong Tsirelson conjecture, is to
determine if $\omega_q(\cl G, \Gamma)$ is always attained. This is essentially the {\it bounded entanglement problem.}

 Also, note that since $C_{qa}(n,c)$ is defined to be
the closure of $C_q(n,c)$ we have that $\omega_q(\cl G, \Gamma) =
\omega_{qa}(\cl G, \Gamma).$

We define the {\it synchronous value of the game given the
  distribution} to be
 \begin{multline*} \omega^s_t(\cl G, \Gamma) =  \\
 \sup \{ \sum_{x,y \in X, i,j \in O}
 \gamma_{x,y} \lambda(x,y,i,j)\gamma_{x,y} p(i,j|x,y) : \big( p(i,j|x,y) \big) \in
 C^s_t(n,c) \}.\end{multline*}

Note that if a game has a winning strategy then $\omega_t(\cl G,
\Gamma) =1,$ for every $\Gamma.$ Conversely, it is easy to see that,
if $\gamma_{x,y} \ne 0$ for all $x,y,$ then for $t \in \{ loc, qa,
 qc \},$ $\omega_t(\cl G, \Gamma) =1$ implies that $\cl G$ has a
 winning strategy.

We summarize a few consequence of our results in these terms.

\begin{prop} If Connes' embedding conjecture is true, then
$\omega_q(\cl G, \Gamma) = \omega_{qc}(\cl G, \Gamma)$ and
$\omega^s_q(\cl G, \Gamma) = \omega^s_{qa}(\cl G, \Gamma)= \omega^s_{qc}(\cl G, \Gamma)$ for every
$\cl G$ and every $\Gamma.$
\end{prop}

If the answer to our Synchronous Approximation Problem is affirmative, 
then $\omega^s_q(\cl G, \Gamma) = \omega^s_{qa}(\cl G, \Gamma)$ for every $\cl G$ and every $\Gamma.$

\begin{prop}  Given a game $\cl G$ with $n$ inputs $X$ and $m$ outputs $O$ and a distribution $\Gamma,$ set 
\[B = \sum_{x,y \in X, i,j \in O} \gamma_{x,y} \lambda(x,y,i,j) e_{x,i}e_{y,j} \in \CFnm .\]
Then
\[ \omega^s_{qc}(\cl G, \Gamma) = \sup \{ \tau(B) \, | \,  \tau: \CFnm \to \bb C \text{ is a tracial state } \},\] and this supremum is attained.
If we restrict the family of traces to those that have finite dimensional
GNS representations, then we obtain $\omega^s_q(\cl G, \Gamma).$ If we restrict the family of traces to those that have abelian GNS representations, then we obtain $\omega^s_{loc}(\cl G, \Gamma),$ and the supremum is attained.
\end{prop}

However, in the finite dimensional case, we can say even more.

\begin{prop}
$\omega^s_q(\cl G, \Gamma) = \sup \{ tr_k\big( \sum_{x,y,i,j} \gamma_{x,y}
\lambda(x,y,i,j) E_{v,i} E_{w,j}\big) \}$
where the supremum is taken over all $k \in \bb N$ and all sets of
projections in $M_k$ satisfying $\sum_{i} E_{v,i} = I.$
\end{prop}
\begin{proof}
Given a finite dimensional representation $\pi$ of $\CFnc$ write
$\pi(\CFnc) = M_{k_1} \oplus \cdots \oplus M_{k_r}$ and
$\pi(e_{v,i}) = E_{v,i,1} \oplus \cdots \oplus E_{v,i,r}$ so that
$\tau(e_{v,i}) = \alpha_1 tr_{k_1}(E_{v,i,1}) + \cdots + \alpha_r
tr_{k_r}(E_{v,i,r})$
for some weights $\alpha_l \ge 0$ with $\alpha_1 + \cdots + \alpha_r =1.$

Note that
\begin{multline*} \tau \big( \sum_{x,y,i,j} \gamma_{x,y}
  \lambda(x,y,i,j) e_{v,i}e_{w,j} \big) = \\
\sum_{l=1}^r \alpha_l\,\, tr_{k_l} \big( \sum_{x,y,i,j} \gamma_{x,y}
\lambda(x,y,i,j) E_{v,i,l}E_{w,j,l} \big),
\end{multline*}
so this convex sum is dominated by 
\[ \max \{ tr_{k_l} \big( \sum_{x,y,i,j} \gamma_{x,y} \lambda(x,y,i,j)
E_{v,i,l}E_{w,j,l} \big): 1 \le l \le r \},\]
from which the result follows.
\end{proof}



\begin{thebibliography}{99}


\bibitem{AHKS06} D. Avis, J. Hagesawa, Y. Kikuchi, and Y. Sasaki,
\newblock A quanutm protocol to win the graph coloring game on all hadamard graphs,
\newblock {\it IEICE Trans. Fundam. electron. Commun. Comput. Sci.,} E89-A(5):1378-1381, 2006.
\newblock arxiv:quant-ph/0509047v4, doi:10.1093/ietfec/e89-a.5.1378.


\bibitem{boca} F. Boca,
\newblock Free products of completely positive maps and spectral sets,
\newblock\textit{J. Funct. Anal.} 97 (1991), 251-263.



\bibitem{cnmsw} P.J. Cameron, M.W. Newman, A. Montanaro, S. Severini and A. Winter,
\newblock On the quantum chromatic number of a graph,
\newblock\textit{The electronic journal of combinatorics} 14 (2007), R81,
\newblock arxiv:quant-ph/0608016.

\bibitem{ce} M. D. Choi and E. Effros,
\newblock Injectivity and operator spaces,
\newblock\textit{J. Funct. Anal.} 24 (1977), 156-209.

\bibitem{C76} A. Connes,
\newblock Classification of injective factors. Cases $II_{1},$ $II_{\infty },$ $III_{\lambda },$ $\lambda \not=1$,
\newblock\textit{Ann. of Math. (2)}, 104, (1976), 73-115.

\bibitem{cmrssw} 
T. Cubitt, L. Man\v{c}inska, D. Roberson, S. Severini, D. Stahlke, and A. Winter,
\newblock Bounds on entanglement assisted source-channel coding via the Lov\'asz $\vartheta$ number and its variants,
\newblock {\it preprint, arxiv:1310.7120v1}, 26 October 2013.


\bibitem{dsw}
R. Duan, S. Severini and A. Winter,
Zero-error communication via quantum channels, non-commutative graphs and a quantum Lov\'{a}sz $\theta$ function,
{\it IEEE Trans. Inf. Theory PP:99 (2012), arXiv:1002.2514v2}.


\bibitem{DJ11}{article}{
  author={Dykema, Ken},
  author={Juschenko, Kate},
  title={Matrices of unitary moments},
  journal={Math. Scand.},
  volume={109},
  year={2011},
  pages={225--239}
}

\bibitem{fkpt_dg}
D. Farenick, A.~S. Kavruk, V.~I. Paulsen and I.~G. Todorov,
\newblock Operator systems from discrete groups,
\newblock\textit{preprint,  arXiv:1209.1152}, 2012.


\bibitem{fkpt}
D. Farenick, A.~S. Kavruk, V.~I. Paulsen and I.~G. Todorov,
\newblock Characterisations of the weak expectation property,
\newblock\textit{preprint, arXiv:1307.1055}, 2013.

\bibitem{fp} D. Farenick and V.I. Paulsen,
\newblock Operator system quotients of matrix algebras and their tensor
  products,
\newblock {\em Math. Scand.} 111 (2012), 210-243.


\bibitem{fritz} T. Fritz,
\newblock Operator system structures on the unital direct sum of C*-algebras,
\newblock\textit{preprint, arXiv:1011.1247}, 2010.


\bibitem{fritz2} T. Fritz,
\newblock Tsirelson's problem and Kirchberg's conjecture,
\newblock\textit{Rev. Math. Phys. 24 (2012), no. 5, 1250012, 67 pp.}



\bibitem{galliardwolf}
V.~Galliard and S.~Wolf, ``Pseudo-telepathy, entanglement, and graph
  colorings,'' in \emph{Proc. IEEE International Symposium on Information
  Theory (ISIT), 2002}, 2002, p. 101.

\bibitem{gr}
C. Godsil and G. Royle,
\newblock Algebraic graph theory,
\newblock\textit{Springer-Verlag, New York,} 2001.

\bibitem{ji}
Z. Ji,
\newblock Binary constraint system games and locally commutative reductions
\newblock {\em preprint (arXiv: 1310.3794)}, 2013.
  
\bibitem{jnppsw}
M. Junge, M. Navascues, C. Palazuelos, D. Perez-Garcia, V. B. Scholtz and R. F. Werner,
\newblock Connes' embedding problem and Tsirelson's problem,
\newblock\textit{J. Math. Physics} 52, 012102 (2011).


\bibitem{krII} R. V. Kadison and J. R. Ringrose,
Fundamentals of the theory of operator algebras II,
{\it American Mathematical Society}, Providence, 1997.


\bibitem{kavruk2011}
A.~S. Kavruk,
\newblock Nuclearity related properties in operator systems,
\newblock {\em preprint (arXiv:1107.2133)}, 2011.


\bibitem{kavruk--paulsen--todorov--tomforde2011}
A.~S. Kavruk, V.~I. Paulsen, I.~G. Todorov, and M.~Tomforde,
\newblock Tensor products of operator systems,
\newblock {\em J. Funct. Anal.} 261 (2011), 267-299.


\bibitem{kptt2010}
A.~S. Kavruk, V.~I. Paulsen, I.~G. Todorov, and M.~Tomforde,
\newblock Quotients, exactness, and nuclearity in the operator system category,
\newblock\textit{Adv. Math.} 235 (2013), 321-360.

\bibitem{lo} L. Lov\'asz,  On the Shannon Capacity of a Graph,
{\it IEEE Transactions on Information Theory,} Vol. II-25, no 1,
January 1979, 1-7.



\bibitem{Ki93}{article}{
  author={Kirchberg, Eberhard},
  title={On non--semisplit extenstions, tensor products and exactness of group C$^*$--algebras},
  journal={Invent. Math.},
  volume={112},
  year={1993},
  pages={449--489}
}


\bibitem{npa}
M. Navascu\'{e}s, S. Pironio, and A. Ac\'{i}n,
\newblock A convergent hierarchy of semidefinite programs characterizing the
  set of quantum correlations.
\newblock \emph{New Journal of Physics}, 10\penalty0 (7):\penalty0 073013,
  2008.


\bibitem{oz} 
N. Ozawa, 
\newblock About the Connes' embedding problem--algebraic approaches,
\newblock {\em preprint (arXiv: 1212.1700)}, 2013.



\bibitem{vpbook} V.~I. Paulsen, Completely Bounded Maps and Operator
  Algebras, {\it Cambridge University Press}, 2002.

  \bibitem{PSSTW} V.~I. Paulsen, S. Severini, D. Stahlke, and A. Winter,
  \newblock Estimating Quantum Chromatic Numbers
  \newblock\textit{preprint, arXiv:1407.6918}, 2014.

\bibitem{pt_chrom}
V. I. Paulsen and I. G. Todorov,
\newblock Quantum chromatic numbers via operator systems,
\newblock\textit{preprint, arXiv:1311.6850}, 2013.




\bibitem{roberson2013variations}
David~E Roberson.
\newblock \emph{Variations on a Theme: Graph Homomorphisms}.
\newblock PhD thesis, University of Waterloo, 2013.


\bibitem{arxiv:1212.1724}
David~E. Roberson and Laura Man\v{c}inska.
\newblock Graph homomorphisms for quantum players,
\newblock\textit{preprint, arXiv:1212.1724}, 2012.

\bibitem{SS12} G. Scarpa and S. Severini,
\newblock Kochen-Specker sets and the rank-1 quantum chromatic number,
\newblock {\it IEEE Trans. Inf. Theory,} 58 (2012), no. 4, 2524-2529.


\bibitem{ws2008} V. B. Scholz, R. F. Werner, 
\newblock Tsirelson's Problem
\newblock preprint, arXiv 0812.4305




\bibitem{sp} D. Spielman, {\it Spectral Graph Theory,} online lecture
  notes, http://www.cs.yale.edu/homes/spielman/561/.

\bibitem{365707}
M.~Szegedy, ``A note on the $\vartheta$ number of Lov\'{a}sz and the
  generalized {D}elsarte bound,'' in \emph{Proc. 35th Annual Symposium on
  Foundations of Computer Science, 1994}, 1994, pp. 36--39.

\bibitem{tsirelson1980}
B.~S. Tsirelson,
\newblock Quantum generalizations of {B}ell's inequality,
\newblock {\em Lett. Math. Phys.}, 4 (1980), no. 4, 93-100.

\bibitem{tsirelson1993}
B.~S. Tsirelson,
\newblock Some results and problems on quantum {B}ell-type inequalities,
\newblock {\em Hadronic J. Suppl.}, 8 (1993), no. 4, 329-345.


\bibitem{V94} Dan Voiculescu,
 \newblock The analogues of entropy and of Fisher's information measure in free probability theory, II,
  \newblock {\it Invent. Math.}, {\bf 118}(1994) 411--440.

\bibitem{V02} Dan Voiculescu,
 \newblock Free entropy,
  \newblock {\it Bull. London Math. Soc.}, {\bf 34}(2002) 257--278.

\end{thebibliography}
\end{document}